%% file: Map_Stratification_From_Toric_Vars.tex
\documentclass[12pt,reqno]{amsproc}
\usepackage[margin=1in]{geometry}
\usepackage{graphicx}
\usepackage[T1]{fontenc}
\usepackage{bbm}
\usepackage{color,xcolor}
\usepackage{subcaption}
\usepackage{mathpazo}
\usepackage{amsmath,amssymb,amsthm}
\newcommand{\spc}{\hspace*{0.25in}}
\usepackage{wrapfig}
\usepackage{tikz-cd}
\usepackage[shortlabels]{enumitem}
\usepackage{bm}
\usepackage{caption}
\usepackage{mathrsfs,bm,nicefrac}
\usepackage[all,cmtip]{xy}

\definecolor{PineGreen}{rgb}{0.0,0.47,0.44}
\definecolor{MidnightBlue}{rgb}{0.1,0.1,0.44}
\definecolor{magenta}{rgb}{1.0,0.0,1.0}
\definecolor{bl1}{HTML}{4479A1}
\definecolor{pur1}{HTML}{52196D}
\definecolor{mag1}{HTML}{2AD0F1}
\definecolor{org1}{rgb}{.92,.39.21}
\definecolor{pur2}{rgb}{.53,.47,.7}

\newcommand{\C}{\mathbb{C}}
\newcommand{\R}{\mathbb{R}}
\newcommand{\CC}{\mathbb{C}}

\newcommand{\VV}{\mathbb{V}}
\newcommand{\N}{\mathbb{N}}
\newcommand{\Z}{\mathbb{Z}}

\newcommand{\sJX}{\mathsf{J}_X}

\newcommand{\cA}{\mathcal{A}}
\newcommand{\cB}{\mathcal{B}}
\newcommand{\cC}{\mathcal{C}}

\newcommand{\rest}[1]{\ensuremath{\left.#1\right|}}
\newcommand{\HJs}{T_J^*}

\DeclareMathOperator{\conv}{conv}

\numberwithin{equation}{section}
\newtheorem{theorem}{Theorem}

\newtheorem{proposition}[theorem]{Proposition}
\newtheorem{example}[theorem]{Example}
\newtheorem{lemma}[theorem]{Lemma}

\newtheorem{definition}[theorem]{Definition}

\usepackage{parskip}
\usepackage{multirow}
\usepackage{booktabs}

\usepackage{hyperref}
\hypersetup{
	colorlinks=true,
	linkcolor=blue,
	citecolor=PineGreen,
	filecolor=magenta,
	urlcolor=blue
}
\usepackage[utf8]{inputenc}

\title{Stratification of Projection Maps From Toric Varieties}
\author{Boulos El Hilany}
\author{Martin Helmer}
\author{Elias Tsigaridas}
\date{\today}

\begin{document}

\begin{abstract}
We prove a combinatorial version of Thom's Isotopy Lemma for projection maps applied to any complex or real toric variety.  
 Our results are constructive and give rise to a method for associating the Whitney strata of the projection to the faces of the polytope of the corresponding toric variety. 
For all examples we produced, our resulting algorithm outperforms known general purpose methods in Helmer and Nanda (FoCM, 2022), and {\DH}inh and Jelonek (DCG, 2021) for computing map-stratifications.
\end{abstract}

\maketitle

\section{Introduction}

A Whitney stratification of a complex variety is its decomposition into simpler components called strata, which are smooth manifolds which fit together in a prescribed (and desirable) way. One can also expand this notion to include analytic maps, which gives rise to a piece-wise fibration. This procedure was first introduced to study singularities of analytic maps and to classify manifolds~ \cite{golubitsky,whitney1965tangents,levine2006singularities}. We can also extend this decomposition to apply to real (semi-)algebraic sets, where the corresponding algorithms are mainly dependent on (variants of) cylindrical algebraic decomposition~\cite{Mather2012};
in this case the applications include, among others, mathematical physics, PDEs, optimization, and data science, e.g., \cite{hnFOCM,pflaum2001analytic,dhinh2019thom, dickenstein2019multistationarity}.

We can express the Whitney stratification of a variety, or a map, using the language of the first order theory of the reals. This results in algorithms for computing the stratification based on quantifier elimination techniques. Their worst case complexity bound is  singly exponential with respect to the number of variables (see e.g., \cite{rannou1991complexity,rannou1998complexity}), which would be an optimal scenario for implementations.
Nevertheless, to the best of our knowledge,
such implementations do not exist. This is so, mainly, due to the nature of quantifier elimination algorithms which rely on infinitesimals for computations. In turn, it contributes to a general lack of practicality of such algorithms. In practice all computer implementations of real quantifier elimination instead rely on cylindrical algebraic decomposition, which both has doubly exponential worst case complexity (with respect to the number of variables) and tends to exhibit this worst case behavior on the majority of examples.

Using a different approach,
{\DH}inh and Jelonek \cite{dhinh2019thom}
proposed a technique to compute the Whitney stratification of a polynomial map over an affine algebraic variety using asymptotic critical values of polynomial maps. Their methods
are based on Gr\"obner basis computations from elimination theory. However, the resulting algorithm becomes almost unusable in practice as the number of  variables needed in intermediate computations is over four times the number of variables appearing in the input.  
Rather recently, Helmer and Nanda \cite{hn,hnFOCM,hnFOCMCorrection} overcame this obstacle by relying on the relation of Whitney stratification with the conormal varieties \cite{le1988limites,FTpolar}. Their algorithms only doubles the number of variables and even though it has double exponential worst case (due to Gr\"obner basis computations) it is efficient in practice;
we refer the interested reader to \cite{hn} for further details and experimental results.
We should also point out that some of the above results were very recently extended to real algebraic varieties~\cite{hnReal}.
We notice that for a restricted class of real varieties, Vorobjov \cite{vorobjov1992effective} suggested a specialized algorithm  for Whitney stratification 
 based on quantifier elimination, that suffers from the same (practical) drawbacks as the similar algorithms for complex varieties. 

The above-mentioned algorithms  apply to all complex and real  algebraic varieties and can be extended to polynomial maps over them. Furthermore, no restriction or assumption on the input is required. As a drawback, however, one is forced to disregard  
any symmetry, combinatorics, or geometry of the objects at hand. 
The latter structure is an important characteristic of
varieties emanating from applications. 
Accordingly, a reasonable approach to accelerate the computationally intensive Whitney stratification algorithms is to study the relationship between strata and structure.

We consider the problem of computing a Whitney stratification of projection maps on toric affine varieties. 
Accordingly, we obtain a 
stratification of the parameter space of any family of affine varieties, whose ideal is generated by binomial equations. Thus, each stratum in the range corresponds to a subset of diffeomorphic binomial varieties. Such varieties appear frequently, for example, in chemistry as mathematical models describing chemical reaction networks~\cite{adamer2019complexity,perez2012chemical}.

A Whitney stratification method for a (real and complex) toric variety $X_{\cA}$ corresponding to any finite set $\cA\subset\Z^{\nu}$ of lattice points already exists in~\cite[Theorem 5.3.~\S 11.5.B]{GKZ}. Namely, each stratum coincides with a $k^*$-orbit (with $k=\R_{\geq 0}$ or $k=\CC$) of $X_{\cA}$. In turn the faces of ${\cA}$ provide an accurate description for the strata of $X_{\cA}$. Accordingly, any projection map $\pi:X_{\cA}\longrightarrow k^m$, which forgets some coordinates, is expected to have a similar stratified description.

We consider any pair $(X_{\cA},~\pi)$ as above and we describe a subdivision of the space $k^m$, for which the restricted projection is a trivial fibration over each stratum. Altogether, the strata are in bijection with the set of faces of $\conv({\cA})$. Our work gives rise to an algorithm that takes $X_{\cA}$ as an input (either using ${\cA}$ or the corresponding binomial representation) and produces a subdivision of $X_{\cA}$ and of $k^m$. The subdivision of $X_{\cA}$ arises from set differences of toric varieties  $X_{\cB}$, for some subset ${\cB}\subset{{\cA}}$ that lies in a face of the polytope $\conv(\cA)$, and the subdivision of $k^m$  arises from the set difference of coordinate subspaces of $k^m$ (of different dimensions) determined by the faces of $\conv({\cA})$.

Our main result, together with the necessary notations, are presented in~\S\ref{sec:main-thm+alg}; it implies that the face-lattice of $\conv(\cA)$ plays an important role in computing the strata of the projection map. 
The running time of our resulting Algorithm~\ref{alg:algorithm} is dominated by (various) combinatorial tasks. Consequently, this provides a significant (overall) run-time speedup. We provide some experimental evidence of this by producing several examples in~\S\ref{sub:algorithm} and making a comparison Table~\ref{tab:runtimes}. 
In \S\ref{sec:main-thm+alg} we also present a detailed Example~\ref{ex:main} to illustrate  Theorem~\ref{thm:stratification_toric_projection} and the functionality of Algorithm~\ref{alg:algorithm}. 

In~\S\ref{sec:decomp-linear} 
we present the proof of our main Theorem~\ref{thm:stratification_toric_projection}; we show how to use the structure of toric ideals to construct diffeomorphisms between fibers of the projection over the same stratum.

\section{Statement of the Main Theorem and Resulting Algorithm}\label{sec:main-thm+alg}

We begin by briefly reviewing the required notations and definitions.

Let $X$ be an algebraic variety of dimension $d$. We say a flag $X_\bullet$ of varieties $X_0\subset \cdots \subset X_d=X$ is a {\em Whitney stratification} of $X$ if, for all $i$, $X_i-X_{i-1}$ is a smooth manifold such that {\em Whitney's condition B} holds for all pairs $M,N$, where $M$ is a connected component of $X_i-X_{i-1}$ and $N$ is a connected component of $X_j-X_{j-1}$. Such connected components are called {\em strata}. A pair of strata, $M,N$ with $\overline{M}\subset \overline{N}$, satisfy Whitney's condition B at a point $x\in M$ with respect to $N$ if: for every sequence $\{p_n\} \subset M$ and every sequence $\{q_n\}\subset N$, with $\lim p_n=\lim q_n=x$, we have $s\subset T$ where $s$ is the limit of secant lines between $p_n,q_n$ and $T$ is the limit of tangent plans to $N$ at $q_n$. We say the pair $M, N$ satisfies condition B if condition B holds, with respect to $N$, at all points $x\in M$.

 A stratification of a continuous map of topological spaces is defined as follows. 
\begin{definition}\label{def:stratmap3}
 Let ${X}_\bullet$ and ${Y}_\bullet$ be Whitney stratifications of topological spaces ${X}$ and ${Y}$. The pair $(X_\bullet, Y_\bullet)$ is a {\bf stratification} of a continuous function $\phi:{X} \to {Y}$ if for each stratum $M \subset {X}$ there exists a stratum $N \subset {Y}$ such that: 
		\begin{enumerate}
			\item the image $\phi(M)$ is wholly contained in $N$; and moreover,
		\item the restricted map $\rest{\phi}_M:M \to N$ is a smooth submersion, i.e.~its derivative $d(\rest{\phi}_S)_x:T_xM \to T_{\phi(x)}N$ is surjective at each point $x$ in $M$.
		\end{enumerate}
	\end{definition}
Our focus will in particular be on the stratification of a special class of polynomial maps. Consider a polynomial map $\phi:X\to Y$ between varieties $X$ and $Y$. {\em Thom's First Isotopy Lemma} \cite[Proposition 11.1]{Mather2012} shows that whenever $\phi$ is proper, the restriction $\rest{\phi}_M:M\to N$ is a $\cC^{\infty}$-fibration for each pair $(M,~N)$ of strata as in Definition~\ref{def:stratmap3}. 
While the maps of interest to us are almost never proper, we still obtain the analog of Thom's First Isotopy Lemma for a certain type of varieties which are given as follows.

 Let $k\in\{\R_{\geq 0},~\C\}$, and consider a parametric system of binomial equations which define a prime ideal. That is, work in a ring $k[x,c]$ where $x=x_1, \dots , x_n$ are thought of as variables and $c=c_1, \dots, c_m$ are thought of as parameters. Let $A\in\Z^{\nu\times (m+n)}$ be the matrix representing a finite subset $\cA\subset\Z^\nu$ having $m+n$ points. Let $I_A=\langle f_1, \dots, f_r\rangle$ be a prime binomial ideal in $k[x,c]$ such that 
\begin{align}\label{eq:toric_definition}
X=X_A= & \VV(I_A)=\overline{\{(t^{a_1}, \dots, t^{a_{n+m}})~|~t\in (k^*)^\nu\}}^k,
\end{align} for some integer $\nu$, denoting the number of rows of $A$, and $a_i$ is the $i^{th}$ column of $A$. The closure in~\eqref{eq:toric_definition} refers to the Zariski closure in the affine variety $k^{m+n}$. 
Then, for each 
parameter choice $p\in k^m$, we obtain a variety $$X(p)=\VV(f_1(x,p), \dots, f_r(x,p))\subset k^n.$$ If $X$ is given as in~\eqref{eq:toric_definition}, it will be called an \emph{affine toric varietty}, whereas the variety $X(p)$ above will be referred to as \emph{affine binomial} -- the latter type specializes to the former whenever, e.g., we have $p=(1,\ldots,1)$. 

Our goal is to study toric varieties given by parametric systems of binomial equations above and obtain a stratification of the parameter space $Y=k_c^m$ such that for each 
stratum $S$ of $Y$, and for any $p,q\in S$, the two binomial varieties $X(p)$ and $X(q)$ are diffeomorphic.

Let $J\subset[n+m]$, and let $T_J$ denote the coordinate subspace of $k^{n+m}$ given by:
\begin{align}\label{eq:torus}
T_J := &~\{y:=(x,c)\in k^{n+m}~|~j\in J\Longrightarrow y_j=0\}.
\end{align} We use the convention $T_{\emptyset} = k^{n+m}$. 
With these notations in hand we now state our main result, the proof of which is postponed to~\S\ref{sec:decomp-linear}.

\begin{theorem}\label{thm:stratification_toric_projection} Let $X$ be
the affine 
toric  
variety in $k^{n+m}$ defined by a prime binomial ideal $I_A=\langle f_1, \dots, f_r\rangle$  in $k[x,c]$, treating $x$ as variables and $c$ as parameters with $k_c^m$ the space of parameters and let $\pi:X\to k^m$ be the coordinate projection $(x,c)\mapsto c$ from $X$ with $Y=\overline{\pi(X)}$. Use the notations above and
consider the two decompositions $X_\bullet:=\{X\cap T_J\}_J$ and $Y_\bullet:=\{\pi(X\cap T_J)\}_J$ of the topological spaces $X$ and $\pi(X)$ respectively. Then the following statements are true:

\begin{enumerate}
	\item\label{it:loc-triv_fib} for each (connected) strata $S$ of $X_{\bullet}$, we have $\rest{\pi}_S$ is a $\cC^{\infty}$-fibration, and
	
	\item\label{it:two_axioms}  the triple $(X_\bullet,~Y_\bullet,~\pi)$ satisfies the two axioms of Definition~\ref{def:stratmap3}.

\end{enumerate}
\end{theorem}

On the one hand, the image of $X$ under the projection $\pi$ is an affine toric variety $X_{\cB}\subset k^m$ for some subset $\cB\subset\cA$. Indeed, the projection retains only the coordinates $(t^{a_i})_{i\in J}$ for some subset $J\subset [m+n]$. On the other hand, it is known (see, e.g.~\cite[Proposition 3.2.9]{CLO}) that, if $\Gamma$ is a face of $\conv(\cA)$, then any toric variety of the form $X_{\Gamma\cap\cA}$ coincides with the restriction of $X$ to a coordinate subspace of $k^{m+n}$; note that $\Gamma \cap \cA$ includes all columns of $\cA$ (written as a matrix) which are contained in $\Gamma$, they need not be vertices of $\Gamma$ nor of $\conv(\cA)$. Therefore, thanks to Theorem~\ref{thm:stratification_toric_projection}, for each stratum of $X$, its projection is a toric variety of the form $X_{\cB}$, where the set $\cB$ is a subset of the set $\Gamma\cap\cA$ given by the coordinates of the projection, i.e.~it is defined by the entries of $\Gamma\cap\cA$ which are retained after projection.

\subsection{Algorithm}\label{sub:algorithm} Using the notations defined above, we derive from Theorem~\ref{thm:stratification_toric_projection} the following algorithm to stratify the projection map $\pi: X\to Y$ from the toric variety $X$ (treating $x$ and $c$ as variables) onto (the closure of) its image $Y=\overline{\pi(X)}$ in the parameter space $k_c^m$. The resulting stratification $(X_\bullet, Y_\bullet)$ is then such that 
the restriction $\rest{\pi}_R:R\to S$ is a $\cC^{\infty}$-fibration for each pair $(R,~S)$ of strata in $(X_\bullet, Y_\bullet)$. 
In particular, the topology of the fiber is constant on each stratum of $Y$.  

As in the discussion following Theorem \ref{thm:stratification_toric_projection} we suppose $X=X_{\cA}$ is the toric variety of interest and $Y=X_{\cB}=\overline{\pi(X)}$ is the toric variety arising from the projection. Note that since $\cB \subset \cA$, then for a face $\Gamma$ of $\conv{\cA}$ the intersection $\cB \cap \Gamma$ is well defined; namely, if $\cB \cap \Gamma=\emptyset$ then the corresponding toric variety  $X_{\cB \cap \Gamma}$ is also empty. Finally, note that, in particular, since the projection is expressed as $(x,c)\mapsto c$, the set $\cB$ corresponds to the last $m$ rows of the matrix $A$ associated to our set $\cA$. 
	\begin{center}
\begin{table}[htbp]  \centering
			
		\begin{tabular}{|r|l|}
			\hline
			~ & {\bf ToricStratMap}$(\pi: X\to Y)$ \\
			\hline
			~&{\bf Input:} The map $\pi: X\to Y$.\\
			~&{\bf Output:} A stratification $(X_\bullet, Y_\bullet)$ of $\pi$. \\
			\hline
			~1 & {\bf Set} $X_\bullet=\{X, \varnothing, \dots, \varnothing\}$, $Y_\bullet=\{Y, \varnothing, \dots, \varnothing\}$\\ 
   ~2 & {\bf For each} $f_i=z^l\pm z^k$ {\bf set} $b_i=l-k$\\
    ~3 & {\bf Set} $V=[v_{1}, \dots, v_r]^T$ be the matrix with rows $v_i$ and {\bf set} $A=\ker V$\\
     	~4 & {\bf Set} $P={\rm conv}(\cA)$\\
      ~5 & {\bf Set} $\cB\subset \cA$ to be subset given by the last $m$ rows of $A$\\
     	~6 & {\bf For each} face $\Gamma$ of $P$ {\bf do:} \\
      	~7 & \spc $d=\dim(X_{\Gamma\cap \cA})$\\
      		~8 & \spc $X_{d}=X_d \cup X_{\Gamma\cap \cA}$\\
        	~9 & \spc $e=\dim(X_{\Gamma\cap \cB})$\\
      		~10 & \spc $Y_e=Y_e \cup X_{\Gamma\cap \cB}$\\
			11 & {\bf Return} $(X_\bullet,Y_\bullet)$\\
			\hline
		\end{tabular}\caption{Algorithm arising from Theorem~\ref{thm:stratification_toric_projection}.}  \label{alg:algorithm}
		\end{table}

	\end{center}\medskip

\subsection{Runtime Comparisons}
We now compare Algorithm~\ref{alg:algorithm} to the general purpose algorithms appearing in \cite{hnFOCMCorrection,helmer2024new} on several instances. We then conclude this section by presenting Example~\ref{ex:main} that illustrates the functionality of Algorithm~\ref{alg:algorithm}.   Since the new algorithm, Algorithm~\ref{alg:algorithm}, 
follows closely 
the combinatorial structure of the $A$-matrix and the simple description of toric varieties, a better performance is expected on examples with large numbers of generators, variables, and higher degrees relative to the general purpose symbolic algorithms which apply to a much broader class of examples.

\begin{table}[htbp]  \centering
  \resizebox{\linewidth}{!}{
    \begin{tabular}{lccc}
        \toprule
          {\bf Input}     \quad~\quad~\quad~\quad~\quad~\quad~\quad~\quad~          &~~~~~~~~\quad~~  {\bf ToricStratMap} ~~~~~ \quad            & \bf Symbolic Alg.~of \cite{hnFOCMCorrection}  ~~~~~  \quad          & \bf Symbolic Alg.~of \cite{helmer2024new} ~~~~~ \quad                 \\
        \midrule
                 $ \pi:\VV(I_1)\to\CC$, see \eqref{eq:TableEx1}                  & 0.08s        &      0.14s  & 0.12s              \\
         $ \pi:\VV(I_2)\to\CC^3$, see \eqref{eq:TableEx2}                  & 0.7s        &      --  & 3.0s              \\
                           $ \pi:\VV(I_3)\to\CC^4$, see \eqref{eq:TableEx3}                  & 3.4s        &      --  & 5.7s              \\
                  $ \pi:\VV(I_4)\to\CC^4$, see \eqref{eq:TableEx4}                  & 26.1s        &      --  & --              \\
            $\pi:\VV(I_5)\to\CC^4$, see \eqref{eq:TableEx5}   & 712.5s   &      --  & --   \\
        \bottomrule
    \end{tabular}}
      \caption{Runtimes for the new combinitorial algorithm on several toric map examples. The runtimes of the symbolic alternative algorithms use version 2.11 of the \texttt{WhitneyStatifications} Macualay2 package. All examples are run in Macualay2 version 1.24.05. In the table, -- denotes examples which take greater than 8 hours to run. Our implementation of {\bf ToricStratMap} used in the tests above is available at: \url{http://martin-helmer.com/Software/WSToric.m2} }
  \label{tab:runtimes}
\end{table}
\medskip
We remark that we do not include the algorithm of {\DH}inh and Jelonek \cite{dhinh2019thom} in Table \ref{tab:runtimes} since the algorithm of \cite{dhinh2019thom} was unable to compute a stratification for the inputs above in under eight hours; in fact the algorithm of \cite{dhinh2019thom} was also not able to compute a Whitney stratification of the Whitney umbrella without any maps involved in 24 hours (see the discussion in~\cite{hnFOCM,hnFOCMCorrection}).

Below we list the defining equations for the examples appearing in Table \ref{tab:runtimes}. As before we consider an ideal $I$ in $k[x_1,\dots, x_n,c_1, \dots, c_m]$ (in our case always a binomial ideal) where we think of the $n$-tuple $x$ as variables and the $m$-tuple $c$ as parameters. Now, if we treat both $x$ and $c$ as variables, this defines a variety $X=\VV(I)\subset k^{n+m}$ and our goal is then to stratify the projection map  $\pi: X\to k^m_c$, $(x,c)\mapsto c$.  
Below we abuse notation by writing this in more compact form as $\pi:\VV(I)\to k^m_c$. Note that in all the examples we consider, the map $\pi$ is surjective onto the parameter space, hence we always have $Y=k^m_c$.

Our first example, $I_1$, is derived from the Whitney umbrella. Examples $I_3$ and $I_5$ come from chemical reaction network theory; the example $I_3$ comes from the two-site kinetic proofreading model, \cite[Example~2.6]{adamer2019complexity}, while $I_5$ comes from the two-site phosphorylation system as given in \cite[Example 3.13]{perez2012chemical}. In the cases arising from chemical reaction networks, the $c_i$s represent (unknown) reaction rate constant parameters while the $x_i$ represent chemical concentrations, see, e.g.~\cite[\S2.3]{adamer2019complexity}, or \cite[\S2]{perez2012chemical}, and the references therein. 
\begin{equation}
    \pi:\VV(I_1)\to\CC_c. \;\; \VV(I_1)\subset \CC^3 ,  \;\; I_1=\langle cx_{2}^{2}-x_{1}^{2}\rangle. \label{eq:TableEx1}
\end{equation}

\begin{equation}
    \pi:\VV(I_2)\to\CC_c^3. \;\; \VV(I_2)\subset \CC^6 ,  \;\; I_2=\langle c_2c_3x_3-c_1x_2,c_3x_1^2-x_2^2,c_1x_1^2-c_2x_2x_3 \rangle. \label{eq:TableEx2}
\end{equation}

\begin{equation}
    \pi:\VV(I_3)\to\CC_c^4. \;\; \VV(I_3)\subset \CC^8 ,  \;\; I_3=\langle c_{3}x_{1}-c_{4}x_{2},\,c_{1}x_{3}x_{4}-c_{2}x_{1}\rangle. \label{eq:TableEx3}
\end{equation}

\scriptsize
    \begin{equation}
        \pi:\VV(I_4)\to\CC_c^4. \;\; \VV(I_4)\subset \CC^9 ,  \;\; I_4=\langle c_{1}x_{3}-x_{4},\,x_{1}^{2}x_{3}-x_{2}^{2}x_{4},\,c_{1}x_{2}^{2}-x_{1}^{2},\,c_{3}x_{3}x_{5}^{2}-x_{1}^{3},\,c_{1}x_{1}^{3}-c_{3}x_{4}x_{5}^{2}\rangle. \label{eq:TableEx4}
    \end{equation}

    \begin{equation}
 \begin{split}       \pi:\VV(I_5) &\to\CC_c^7. \;\; \VV(I_5)\subset \CC^{16} ,  \\    I_5= &\langle c_{3}x_{5}-c_{2}x_{7},\,c_{5}x_{4}-c_{4}x_{6},\,c_{3}x_{3}x_{9}-c_{6}x_{7},\,c_{2}x_{3}x_{9}-c_{6}x_{5},\,c_{5}x_{2}x_{9}-c_{7}x_{6},\,c_{4}x_{2}x_{9}-c_{7}x_{4},\,c_{6}x_{2}x_{8}-c_{6}x_{5},\\ &c_{3}x_{2}x_{8
     }-c_{2}x_{7},\,c_{2}x_{2}x_{8}-c_{2}x_{5},\,c_{7}x_{1}x_{8}-c_{1}x_{2}x_{9},\,c_{5}x_{1}x_{8}-c_{1}x_{6},\,c_{4}x_{1}x_{8}-c_{1}x_{4},\,c_{6}c_{7}x_{6}x_{8}-c_{5}c_{6}x_{5}x_{9},\\ &c_{3}c_{7}x_{6}x_{8}-c_{2}c_{5}x_{7}x_{9},\,c_{2}c
     _{7}x_{6}x_{8}-c_{2}c_{5}x_{5}x_{9},\,c_{6}c_{7}x_{4}x_{8}-c_{4}c_{6}x_{5}x_{9},\,c_{3}c_{7}x_{4}x_{8}-c_{2}c_{4}x_{7}x_{9},\\ &c_{2}c_{7}x_{4}x_{8}-c_{2}c_{4}x_{5}x_{9},\,c_{3}c_{7}x_{3}x_{6}-c_{5}c_{6}x_{2}x_{7},\,c_{1}c_{3}x_{2}x
     _{6}-c_{2}c_{5}x_{1}x_{7},\,c_{5}c_{6}x_{2}x_{5}-c_{2}c_{7}x_{3}x_{6},\\ &c_{5}c_{6}x_{1}x_{5}-c_{1}c_{6}x_{2}x_{6},\,c_{2}c_{5}x_{1}x_{5}-c_{1}c_{2}x_{2}x_{6},\,c_{3}c_{7}x_{3}x_{4}-c_{4}c_{6}x_{2}x_{7},\,c_{2}c_{7}x_{3}x_{4}-c_{4}
     c_{6}x_{2}x_{5},\\ &c_{1}c_{6}x_{2}x_{4}-c_{4}c_{6}x_{1}x_{5},\,c_{1}c_{3}x_{2}x_{4}-c_{2}c_{4}x_{1}x_{7},\,c_{1}c_{2}x_{2}x_{4}-c_{2}c_{4}x_{1}x_{5},\,c_{1}c_{6}x_{2}^{2}-c_{2}c_{7}x_{1}x_{3}
 \rangle.\end{split} \label{eq:TableEx5}
    \end{equation}
    \normalsize

\begin{example}\label{ex:main}
Let $k\in\{\R_{\geq 0},\C\}$. Consider a family $\cC:=\{C_c\}_{c\in k}\subset k^2$ of plane curves given by the vanishing locus of the ideals of the form 
\begin{equation}\label{eq:umbrella}
    cy^2-x^2. 
\end{equation} \begin{figure}[h]
    \centering
    \includegraphics[width=0.3\linewidth]{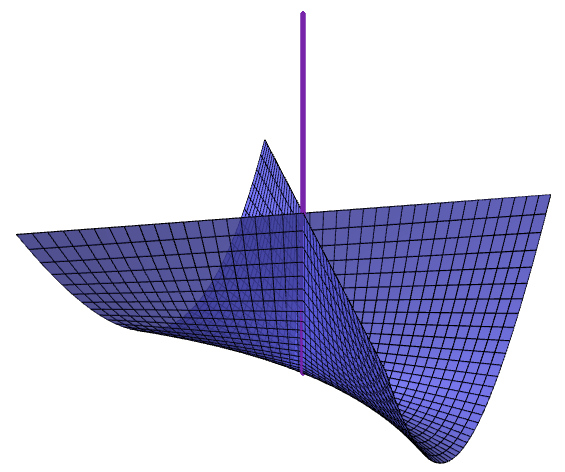}
    \caption{The Whitney umbrella $W\subset\R^3$ obtained from~\eqref{eq:umbrella}; it contains the  real toric variety $X=W\cap\R_{\geq 0}^3$ from Example~\ref{ex:main}.}
    \label{fig:umbrella}
\end{figure}
Then, the curve $C_c$ is diffeomorphic to $\{x = \pm y\}\cap k^2$ for any $c\neq 0$ in $k$, whereas $C_0$ has no other representative in $\cC$. 

This observation can be confirmed from a different point of view: The family $\cC$ consists of fibers $\left.\pi^{-1}\right|_{X}(c)$ under the map $\pi:X\longrightarrow k$, $\pi :(x,y,c)\mapsto c$, where $X\subset k^3$ is the vanishing locus of the ideal~\eqref{eq:umbrella} in $\Z[x,y,c]$ (see Figure \ref{fig:umbrella}). The latter is a toric ideal, whose $A$-matrix is 
\[
A:= ~ \begin{bmatrix}
 1&   0 & 2  \\
  1&  1 & 0 
\end{bmatrix}=\begin{bmatrix}
    a_1 & a_2 & a_3
\end{bmatrix}.
\] 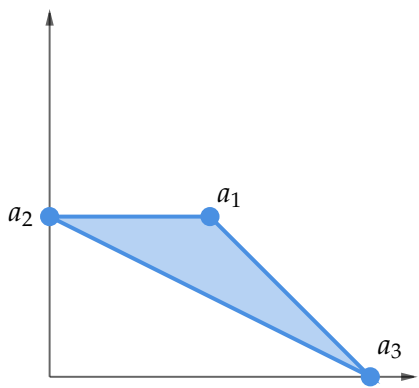
\begin{figure}[h]
    \centering
    \input{fig-triangle}
    \caption{The set $\conv(\cA)\subset\R^2$ corresponding to Example~\ref{ex:main}.}
    \label{fig:triangle}
\end{figure} We can decompose $X$ into seven strata, each of which corresponds to a face of $P=\conv(\cA)$ (see Figure~\ref{fig:triangle}). The polytope $P$ has 3 vertices, which are the columns $a_1, a_2, a_3$ of the matrix $A$, and three edges $[a_3,a_1]=\conv{(a_3, a_1)}$, $[a_2,a_1]=\conv{(a_2, a_1)}$, and $[a_2,a_3]=\conv{(a_2, a_3})$. The collection of edges $[a_3,a_1]$, $[a_2,a_1]$, and $[a_2,a_3]$, gives rise to the affine toric varieties (which are all copies of $k$):
\begin{align*}
 X_{[a_3,a_1]}:= & ~X\cap \{y=0\}=\VV(x,y), \\ 
 X_{[a_2,a_1]}:= & ~X\cap \{c=0\}=\VV(x,c), \\
 X_{[a_2,a_3]}:= & ~X\cap \{x=0\}=\VV(x,y)\cup \VV(x,c).  
\end{align*} On the codomain side, the edge $[a_3,a_1]$ gives a copy of $Y=k$ (as does the whole polytope $P$), the edge $[a_2,a_1]$ gives a copy of $\VV(c)\subset Y$, and the edge $[a_2,a_3]$ gives another copy of $Y=k$. The collection of vertices $a_{1}$, $a_{2}$, and $a_{3}$, give rise to 
\begin{align*}
     X_{a_1}:= & ~X\cap \{y=c=0\}=\VV(x,y,c)=\{(0,0,0)\},\\
 X_{a_2}:= & ~X\cap \{x=c=0\}=\VV(x,c), \\ 
 X_{a_3}:= & ~X\cap \{x=y=0\}=\VV(x,y).
\end{align*}On the codomain side, $a_1$ and $a_2$ give copies of $\VV(c)\subset Y$, while $a_3$ gives a copy of $Y$. 

Hence the stratification $(X_\bullet, Y_\bullet)$ of $\pi$ is given by $X_\bullet=(X_2, X_1, X_0)$, $Y_\bullet=(Y_1, Y_0)$, where $X_2=X_A=X$, $$X_1= X_{[a_3,a_1]} \cup  X_{[a_2,a_1]} \cup  X_{[a_2,a_3]}\cup X_{a_2} \cup X_{a_3} =\VV(x,y)\cup \VV(x,c),$$ $X_0= X_{a_1}$,  $Y_1=Y=k$ and $Y_0=\VV(c)=0\subset k$.
Therefore, we in particular obtain the stratification $Y_\bullet=\{\emptyset\subset\{0\}\subset k\}$ of $k$. Our main result states that over each connected component of $k^*=k-0$ the map $\left.\pi\right|_{X}$ is a fibration.
\end{example}

\section{Proof of the Main Result and Associated Lemmas}\label{sec:decomp-linear}
This section is devoted to proving Theorem~\ref{thm:stratification_toric_projection}. In what follows, we retain all notations from~\S\ref{sec:main-thm+alg}. However, when considering a (toric) variety $X$, rather than working with closed sets when defining stratifications $X_\bullet$, and defining strata via set difference, we instead work explicitly with the resulting (connected) open strata. We will abuse notation slightly and write $S\in X_\bullet$ when $S$ is an open connected strata of $X$. 

Recall that $X\subset k^{m+n}$ is an affine toric variety with corresponding set of lattice points $\cA\subset\Z^{\nu}$. We shall also need to define some additional notations. In what follows, we fix a set $\cA:=\{a(1),\ldots,a(n+m)\}\subset\Z^{\nu}$, and consider the following sequence of functions
\begin{equation}\label{eq:sequence-1}
(k^*)^\nu\overset{\Phi_A}{\longrightarrow} (k^*)^{n+m}\overset{\pi}{\longrightarrow}  (k^*)^m,
\end{equation} where $\pi$ is a projection $(y_1,\ldots,y_{n+m})\longmapsto (y_{n+1},\ldots,y_{n+m})$, and $\Phi_A$ is the toric map $t\longmapsto(t^{a})_{a\in\cA}$.

Let $J\subset[n+m]$, and let $\HJs$ denote the sub-torus of $k^{n+m}$ given by:
\begin{align}\label{eq:torus}
\HJs := &~\{y:=(x,c)\in k^{n+m}~|~y_j=0\Longleftrightarrow j\in J\}.
\end{align} Note that $T_{\emptyset} = (k^*)^{n+m}$. Let $X^*$ denote the intersection $
X^*:=  X\cap T_{\emptyset}$, and consider the map 
\begin{align}\label{eq:map}
F:=\rest{\pi}_{X^*}: &~X^*\longrightarrow (k^*)^m,
\end{align} where $\pi:X\to k^m$ is the restriction of the projection $(x,c)\mapsto c$ to $X$.

We may now proceed with the proof of Theorem~\ref{thm:stratification_toric_projection}. A key ingredient in the proof of Theorem~\ref{thm:stratification_toric_projection} is the following result, which will be proven later in~\S\ref{sub:proof_thm2}.

\begin{proposition}\label{prp:fibration_torus}
The map $F$ is a $\cC^{\infty}$-fibration.
\end{proposition} 

\subsection{Proof of Theorem~\ref{thm:stratification_toric_projection}}\label{sub:proof_thm_main} We first introduce the following result, which follows directly from~\cite[Proposition 2.1.8]{CLO}. 

\begin{lemma}\label{lem:known_theorem}
If $X\cap \HJs\neq\emptyset$, then we have that
\begin{align*}
X\cap \HJs = &~\Phi_{\cB}((k^*)^q),
\end{align*} where $q:=\dim \Gamma$, and $\cB:=\Sigma\cap \Z^{\nu}$ for some face $\Gamma$ of $\conv{\cA}$ such that.
\end{lemma}

The first Item of Theorem~\ref{thm:stratification_toric_projection} follows directly from Lemma~\ref{lem:known_theorem} and Proposition~\ref{prp:fibration_torus}.

Now, we show Item~\eqref{it:two_axioms}. Projecting each element in $X_\bullet$ gives rise to a decomposition $Y_\bullet$ of $\pi(X)$ that satisfies the first axiom in Definition~\ref{def:stratmap3}. 

To prove the second axiom of Definition~\ref{def:stratmap3}, we proceed as follows. Thanks to Lemma~\ref{lem:known_theorem}, it is enough to show that $\left. \pi\right|_{X\cap\HJs}$ is a submersion only for the case where $J=\emptyset$ (i.e., we only consider the stratum $X^*\in X_\bullet$). 

To show that $F$ is a submersion, we compute the Jacobian matrix of the pair $(X^*,~F)$. This is a $(m+r)\times (n+m)$--matrix $A(x)=[\sJX~|~M]$,  
where $\sJX$ is an $r\times (n+m)$--matrix, and $M$ is a $m\times (n+m)$--matrix, given as

\begin{tabular}{cccc}
$
\sJX :=	\begin{bmatrix}
{\partial f_1}/{\partial y_1} &\cdots& {\partial f_r}/{\partial y_1}\\
 \vdots& \ddots & \vdots  \\
	{\partial f_1}/{\partial y_{n+m}} &\cdots& {\partial f_r}/{\partial y_{n+m}}\\
	\end{bmatrix},
	$
& 
and
&

$
M := \begin{bmatrix}0 & \cdots &0 \\
\vdots & \ddots &\vdots \\
0 & \cdots &0 \\
1&\cdots &0 \\
\vdots &\ddots &\vdots \\
	0& \cdots &1 \\
	\end{bmatrix}.
$
\end{tabular}

Since $X^*\subset(k^*)^{m+n}$, the generators of $I(X^*)$ are binomials that can be expressed as 
\begin{align*}
f_i = &  ~y^{\beta(i)} - 1,~i=1,\ldots,r,
\end{align*} for some $\cB:=\{\beta(1),\ldots,\beta(r)\}\in\Z^{n+m}$. Therefore, we get
\begin{equation}\label{eq:sJX}
\sJX =	\begin{bmatrix}
{\beta_1(1)~y^{\beta(1)}}/{y_1}&\cdots& {\beta_1(r)~y^{\beta(r)}}/{y_1}\\
 \vdots& \ddots & \vdots  \\
{\beta_{m+n}(1)~y^{\beta(1)}}/{y_{m+n}}&\cdots& {\beta_{m+n}(r)~y^{\beta(r)}}/{y_{m+n}}\\
	\end{bmatrix} = y^{\beta(1)}\cdots y^{\beta(r)}/( y_1\cdots y_{m+n}) \cdot B,
\end{equation} where $\beta(i):=(\beta_1(i),\ldots,\beta_{n+m}(i))$, and $B$ is the matrix expression of $\cB$. Then, for all $y\in (k^*)^{n+m}$ the rank of $\sJX$ is independent of $y$. 

Hence, $A(x)$ has full rank for all $y\in (k^*)^{n+m}$. This yields the proof of Theorem~\ref{thm:stratification_toric_projection}.

\subsection{Very affine toric projections are fibrations}\label{sub:proof_thm2} This part is devoted to proving Proposition~\ref{prp:fibration_torus}. The following two key lemmas will be proven at the end of this section.

\begin{lemma}\label{lem:f-proper}
If $\nu\leq m$, then the map $F$ is proper.
\end{lemma}

\begin{lemma}\label{lem:diffeomorph}
If $\nu > m$, then for each $c\in (k^*)^m$, we have that $F^{-1}(c)$ is diffeomorphic to $F^{-1}(1,\ldots,1)$. 
\end{lemma}

\begin{proof}[Proof of Proposition~\ref{prp:fibration_torus}]
    Since the proof of Item~\eqref{it:two_axioms} of Theorem~\ref{thm:stratification_toric_projection} does not require Proposition~\ref{prp:fibration_torus} (see~\S\ref{sub:proof_thm_main}), we can use it here. Then, we have that $F$ is a submersion. Consequently, thanks to Lemma~\ref{lem:f-proper}, we can apply Thom's First Isotopy Lemma to deduce that $F$ is a locally trivial fibration whenever $\nu\leq m$. Furthermore, Lemma~\ref{lem:diffeomorph} implies that $F$ is a locally trivial fibration whenever $\nu > m$. This yields the proof of Proposition~\ref{prp:fibration_torus}.
\end{proof} 

\begin{proof}[Proof of Lemma~\ref{lem:f-proper}]

For every $c\in (k^*)^m$, we have
\begin{multline}\label{eq:pre-image}
F^{-1}(c) = \{
	(x,~c)\in (k^*)^n\times(k^*)^m~|~\exists~t\in (k^*)^\nu,~x_i ~=~t^{a(i)}\text{ for } i=1,\ldots,n	\text{, and } \\
c_j ~=~t^{a(n+j)}\text{ for } j=1,\ldots,m\}.
\end{multline}
Assume first that $\nu=m$. On the one hand, there exists $N\in\N$ such that for each $c\in(k^*)^m$, the system 
\begin{align}\label{eq:sys-c_t}
c_j = & ~t^{a(n+j)}\text{ for } j=1,\ldots,m,
\end{align} has $N$ solutions $t\in(k^*)^m$. Indeed, the Jacobian matrix in $t$ of this system is a matrix $\hat{A}$ scaled by a monomial in $t$, and the rows are vectors $a(i)$ (see e.g., proof of Theorem~\ref{thm:stratification_toric_projection}). Since the original subset $\cA\subset\Z^{\nu}$ spans the space $\R^{\nu}$, and the values $c_i$ are non-negative for $k=\R_{\geq 0}$, the number of solutions $t\in(k^*)^m$ of the system~\eqref{eq:sys-c_t} is finite and independent of $c$.

On the other hand, the Jacobian matrix in $t$ to the system 
\begin{align}\label{eq:sys-x_t}
x_i = & ~t^{a(i)}\text{ for } i=1,\ldots,n,
\end{align} has also full rank due to the same reasons.  

Therefore, the $N$ solutions to~\eqref{eq:sys-c_t} give rise to exactly $N$ distinct points $x\in (k^*)^n$. Hence, we get $\# F^{-1}(c)=N$ for any $c\in (k^*)^m$. This implies that $F$ is finite, hence, proper.

Assume now that $\nu < m$. Then, consider a toric map $\Phi_{\cB}:(k^*)^\nu\times (k^*)^{m-\nu}\longrightarrow (k^*)^{n+m}$, $(t,~s)\longmapsto (t^{a(1)}s^{b(1)},\ldots,t^{a(n+m)}s^{b(n+m)})$, for some subset $b(1),\ldots,b(n+m)\in\Z^{m-\nu}$. We may choose $\cB$ so that the subset 
\[
\cC:=\{(a(1),b(1)),\ldots,(a(n+m),~b(n+m))\}\subset\Z^{\nu}\times\Z^{m-\nu}
\] is non-degenerate. Then, the following inclusion holds
\[
X_{\cA}:=\Phi_{\cA}((k^*)^\nu) ~=~ \Phi_{\cB}((k^*)^\nu\times\{(1,\ldots,1)\})~\subset~ \Phi_{\cC}((k^*)^m)=:X_{\cC}.
\] From the first part of this proof, we have that $\left. \pi\right|_{X_{\cC}}$ is proper. Then, for any affine subvariety $S\subset X_{\cC}$, the map $\rest{\pi}_{S}$ is also proper. Since $X_{\cA}\subset X_{\cC}$, it follows that $F$ is proper.
\end{proof}

\begin{proof}[Proof of Lemma~\ref{lem:diffeomorph}]
Fix any point $c\in (k^*)^m$. From the expression of $F^{-1}(c)$ in~\eqref{eq:pre-image}, we have the following relation
\begin{align*}
I(F^{-1}(c)) = &~ \left\langle x_1 - t^{a(1)},\ldots,~x_n - t^{a(n)},~c_1 - t^{a(n+1)},\ldots,~c_m - t^{a(n+m)}\right\rangle\cap k[x_1,\ldots,x_n]. 
\end{align*} Hence, this ideal is generated by binomials in $x$. Namely, there is a subset 
\[
\{\alpha(1),\beta(1),\ldots,\alpha(\nu-m),\beta(\nu-m)\}\subset\N^n
\] and monomial functions $\varphi_1,\ldots,\varphi_{\nu-m}:(k^*)^m\longrightarrow k^*$ satisfying
\begin{align*}
I(F^{-1}(c)) = &~ \left\langle x^{\alpha(1)} - \varphi_1(c)~x^{\beta(1)},\ldots,~x^{\alpha(\nu -m)} - \varphi_{\nu -m}(c)~x^{\beta(\nu -m)}\right\rangle. 
\end{align*}
Let $s\in (k^*)^n$, such that we have
\begin{align*}
s^{\alpha(i) - \beta(i)} = & ~\varphi_i(c),\quad i=1,\ldots,\nu - m.
\end{align*} The existence of $s$ is a result of solving the above system of monomial equations in $n$ variables with $n>\nu - m$; if $k=\R_{\geq 0}$, such solution will be in the positive orthant. Then, it holds that
\begin{align*}
(s\cdot x)^{\alpha_i} - (s\cdot x)^{\beta_i} = 0~\Longleftrightarrow~x^{\alpha_i} - \varphi_i(c)~ x^{\beta_i} = 0,\quad i=1,\ldots,\nu - m,
\end{align*} where $s\cdot x := (s_1x_1,\ldots,~s_nx_n)$. We conclude that the automorphism $x\mapsto s\cdot x$ sends $F^{-1}(1,\ldots,1)$ to $F^{-1}(c)$.
\end{proof}

\bibliographystyle{abbrv}
	\bibliography{library}
\end{document}

%% file: fig-triangle.tex
\tikzset{every picture/.style={line width=0.75pt}} 

\begin{tikzpicture}[x=0.75pt,y=0.75pt,yscale=-1,xscale=1]

\draw [color={rgb, 255:red, 0; green, 0; blue, 0 }  ,draw opacity=0.6 ]   (129,243.67) -- (312.67,243.67) ;
\draw [shift={(314.67,243.67)}, rotate = 180] [fill={rgb, 255:red, 0; green, 0; blue, 0 }  ,fill opacity=0.6 ][line width=0.08]  [draw opacity=0] (8.4,-2.1) -- (0,0) -- (8.4,2.1) -- cycle    ;
\draw [color={rgb, 255:red, 0; green, 0; blue, 0 }  ,draw opacity=0.6 ]   (129,243.67) -- (129,60) ;
\draw [shift={(129,58)}, rotate = 90] [fill={rgb, 255:red, 0; green, 0; blue, 0 }  ,fill opacity=0.6 ][line width=0.08]  [draw opacity=0] (8.4,-2.1) -- (0,0) -- (8.4,2.1) -- cycle    ;
\draw  [color={rgb, 255:red, 74; green, 144; blue, 226 }  ,draw opacity=1 ][fill={rgb, 255:red, 74; green, 144; blue, 226 }  ,fill opacity=1 ] (286.4,243.67) .. controls (286.4,241.27) and (288.34,239.33) .. (290.73,239.33) .. controls (293.13,239.33) and (295.07,241.27) .. (295.07,243.67) .. controls (295.07,246.06) and (293.13,248) .. (290.73,248) .. controls (288.34,248) and (286.4,246.06) .. (286.4,243.67) -- cycle ;
\draw  [color={rgb, 255:red, 74; green, 144; blue, 226 }  ,draw opacity=1 ][fill={rgb, 255:red, 74; green, 144; blue, 226 }  ,fill opacity=1 ] (124.67,162.8) .. controls (124.67,160.41) and (126.61,158.47) .. (129,158.47) .. controls (131.39,158.47) and (133.33,160.41) .. (133.33,162.8) .. controls (133.33,165.19) and (131.39,167.13) .. (129,167.13) .. controls (126.61,167.13) and (124.67,165.19) .. (124.67,162.8) -- cycle ;
\draw  [color={rgb, 255:red, 74; green, 144; blue, 226 }  ,draw opacity=1 ][fill={rgb, 255:red, 74; green, 144; blue, 226 }  ,fill opacity=1 ] (205.53,162.8) .. controls (205.53,160.41) and (207.47,158.47) .. (209.87,158.47) .. controls (212.26,158.47) and (214.2,160.41) .. (214.2,162.8) .. controls (214.2,165.19) and (212.26,167.13) .. (209.87,167.13) .. controls (207.47,167.13) and (205.53,165.19) .. (205.53,162.8) -- cycle ;
\draw  [color={rgb, 255:red, 74; green, 144; blue, 226 }  ,draw opacity=1 ][fill={rgb, 255:red, 74; green, 144; blue, 226 }  ,fill opacity=0.4 ][line width=1.5]  (209.87,162.8) -- (290.73,243.67) -- (129,162.8) -- cycle ;

\draw (122.67,162.8) node [anchor=east] [inner sep=0.75pt]  [font=\small]  {$a_2$};
\draw (292.73,235.93) node [anchor=south west] [inner sep=0.75pt]  [font=\small]  {$a_3$};
\draw (211.87,159.4) node [anchor=south west] [inner sep=0.75pt]  [font=\small]  {$a_1$};

\end{tikzpicture}